\documentclass{amsart}
\usepackage[utf8]{inputenc}

\usepackage{amssymb,tikz}
\usepackage{latexsym}
\usepackage{amsmath,mathdots}
\usepackage{euscript}
\usepackage{graphics}
\usepackage{todonotes}
\usepackage{pgf,tikz,pgfplots}
\pgfplotsset{compat=1.14}
\usepackage{mathrsfs}
\usetikzlibrary{arrows}
   \def\dN{{\mathbb N}}   
      \def\dR{{\mathbb R}}

\def\bm\chi{\mbox{\boldmath$\chi$}}

\let\xker=\ker \def\ker{{\xker\,}}

\def\deg{\operatorname{deg}}

\newtheorem{theorem}{Theorem}[section]
\newtheorem{proposition}[theorem]{Proposition}
\newtheorem{corollary}[theorem]{Corollary}
\newtheorem{lemma}[theorem]{Lemma}

\theoremstyle{remark}

\newtheorem{remark}[theorem]{Remark}

\numberwithin{equation}{section}
\begin{document}
	
	\date{\today}
	
	\author[M.~Derevyagin]{Maxim~Derevyagin}
	\address{
		MD,
		Department of Mathematics\\
		University of Connecticut\\
		341 Mansfield Road, U-1009\\
		Storrs, CT 06269-1009, USA}
	\email{derevyagin.m@gmail.com}
	
	\author[A.~Minenkova]{Anastasiia~Minenkova}
	\address{
		AM,
		Department of Mathematics\\
		University of Connecticut\\
		341 Mansfield Road, U-1009\\
		Storrs, CT 06269-1009, USA}
	\email{anastasiia.minenkova@uconn.edu}
	
	\author[N.~Sun]{Nathan~Sun}
	\address{
		NS,
		Department of Mathematics\\
		Harvard University\\
		86 Brattle Street\\
		Cambridge, MA, 02138, USA}
	\email{nathan99sun@gmail.com}


	%%%%%%%%%%%%%%%%%%%%%%%%%%%%%%%%%%%%%%%%%%%%%%%%%%%%%%%%%%%%%%%%%%%%%%%%%%%%%%%%%%

	\title[The Serret theorem and perfect quantum state transfer]{A theorem of Joseph-Alfred Serret and its relation to perfect quantum state transfer}
	
	\begin{abstract}
		In this paper we recast the Serret theorem about a characterization of palindromic continued fractions in the context of polynomial continued fractions. Then, using the relation between symmetric tridiagonal matrices and polynomial continued fractions we give a quick exposition of the mathematical aspect of the perfect quantum state transfer problem.
	\end{abstract}
	
	\maketitle

	\section{Introduction}
	
	Let $p$ and $q$ be {\color{black}{relatively prime }}positive integers such that $q<p$. It is not so hard to see that every positive rational number $q/p<1$ can be uniquely represented in the following manner
	\begin{equation}\label{nCF}
		\frac{q}{p}=\cfrac{1}{\mathfrak{a}_0+\cfrac{1}{\mathfrak{a}_1+\cfrac{1}{\ddots+\cfrac{1}{\mathfrak{a}_N}}}},
	\end{equation}
	where $\mathfrak{a}_i\in\dN$ for $i=0$, $1$, \dots, $N-1$, $\mathfrak{a}_N\ge 2$ is integer,  and $\dN$ is the set of all natural numbers, that is, $\dN=\{1,2,3,\dots\}$. Indeed, it is just another way to write down the outcome of the Euclidean algorithm applied to the pair of numbers $p$ and $q$, which determines the set of quotients $\mathfrak{a}_0$, \dots, $\mathfrak{a}_N$ uniquely.
	{\color{black}{To be more precise, for $p$ and $q$ there exists a unique pair of numbers $a_0$ and $r$ such that
			\[
			p=a_0q+r, \qquad 0<r<q.\]
			The latter can be written as
			\[
			\frac{q}{p}=\cfrac{1}{a_0+\cfrac{r}{q}}.\]
			Evidently, $q$ and $r$ have no common positive factors other than 1 and so we can apply the same reasoning to $q$ and $r$. Since the procedure is clearly finite, we will end up with \eqref{nCF}.}}
	
	The right-hand side of \eqref{nCF} is called a continued fraction. To be more precise, it is a finite continued fraction and to such a continued fraction one can associate a finite sequence of its convergents
	\begin{equation}\label{nConvergents}
		\frac{q_k}{p_k}=\cfrac{1}{\mathfrak{a}_0+\cfrac{1}{\ddots+\cfrac{1}{\mathfrak{a}_{k-1}}}}, \quad k=1, \dots, N+1,
	\end{equation}
	where, in particular, $p_1=\mathfrak{a}_0$, $q_1=1$ and $p_{N+1}=p$, $q_{N+1}=q$.
	
	One of the fundamental properties of continued fractions is that the sequences $q_k$ and $p_k$ satisfy the three-term recurrence relations
	\begin{align}
		\label{qRec} q_{k+1}&=\mathfrak{a}_{k}q_{k}+q_{k-1}\\
		\label{pRec} p_{k+1}&=\mathfrak{a}_{k}p_{k}+p_{k-1}
	\end{align}
	for $k=2$, \dots, $N$. {\color{black} One can easily see the validity of the relations \eqref{qRec} and \eqref{pRec} when $k=2$ directly from \eqref{nConvergents} provided we set $p_0=1$ and $q_0=0$. Indeed, if $k=2$, formula \eqref{nConvergents} gives 
		\[p_2=\mathfrak{a}_0\mathfrak{a}_1+1=\mathfrak{a}_1p_1+p_0,\quad q_2=\mathfrak{a}_1=\mathfrak{a}_1q_1+q_0.\]} The rest can be proved using the mathematical induction (for more details, see \cite[pages 4 and 5]{Khinchin}).  
	
	As a matter of fact, formulas \eqref{qRec} and \eqref{pRec} tell us that the sequences $q_k$ and $p_k$ are two linearly independent solutions of the second-order difference equation
	\begin{equation}\label{nDE}
		u_{k+1}-\mathfrak{a}_{k}u_{k}-u_{k-1}=0, \quad k=0, 1, 2, \dots.
	\end{equation}
	Namely, the unique solution to \eqref{nDE} that satisfies the initial conditions
	\[
	u_{-1}=0, \quad u_0=1
	\]
	is the sequence $p_k$, that is, in this case $u_k=p_k$ for $k=1$, $2$, \dots, $N+1$ {\color{black}and, thus, $p_{-1}=0$ and  $p_0=1$}. Also, the unique solution to \eqref{nDE} that satisfies the initial conditions
	\[
	{\color{black}{u_{-1}=1}}, \quad u_0=0
	\]
	is the sequence $q_k$, i.e. in this case $u_k=q_k$ for $k=1$, $2$, \dots, $N+1$ {\color{black}and, to be consistent,  we set $q_{-1}=1$ and $q_0=0$}.
	
	Equation \eqref{nDE} is a discrete version of the $1D$ Scr\"odinger equation and the Serret theorem we discuss in this note gives a criterion for the potential function of such second order difference operator to be mirror-symmetric. The latter means that the sequence $\mathfrak{a}_{k}$ possess the property
	\[
	\mathfrak{a}_{k}=\mathfrak{a}_{N-k}
	\]
	for all $k\in\{0, 1, 2, \dots, N\}$. We will also consider some generalizations of the Serret theorem in order to demonstrate its relation to the problem of perfect quantum state transfer that has recently attracted a lot of interest (for instance, see \cite{DDMT}, \cite{Kay10}, \cite{VZh12}).
	
	The paper is organized as follows.  In the next section we give some basic properties of continued fractions and prove the Serret theorem. In Sections 3 and 4 we extend the theory to the case of $J$-fractions, which are a specific case of polynomial continued fractions, \textcolor{black}{and, in particular, show that the Serret theorem still holds in this case}. Then, in Section 5 we consider the relation between $J$-fractions and Jacobi matrices. After that, we give an exposition of the perfect state transfer problem in Section 6. Finally, Section 7 discusses a generalization of the Serret theorem to a general polynomial continued fraction \textcolor{black}{and, thus, it demonstrates that Serret's result is valid for the ring of polynomials}.

	\section{The Serret theorem about palindromic continued fractions}
	
	In this section we prove the Serret theorem. Before we can proceed, we need to prove an auxiliary result, which contains important formulas in the theory of continued fractions and they actually hold for any continued fraction \eqref{nCF}.
	
	\begin{proposition}\label{contfrac} Let $\{p_k\}$ and $\{q_k\}$ be the sequences  defined by \eqref{nConvergents}. Then 
		\begin{itemize}
			\item [(i)]we have that \begin{equation}\label{nReverseCF}
				\frac{p_{k-1}}{p_k}=\cfrac{1}{\mathfrak{a}_{k-1}+\cfrac{1}{\ddots+\cfrac{1}{\mathfrak{a}_{0}}}}
			\end{equation}
			for $k=1$, $2$, \dots, $N+1$ and where we set $p_0=1$ for convenience.
			\item[(ii)] The Wronskian of the sequences $p_k$ and $q_k$ is either 1 or -1. More precisely, the relation is given by
			\begin{equation}\label{minus1}
				\det\begin{pmatrix}p_{k+1}&p_{k}\\q_{k+1}&q_{k}
				\end{pmatrix} = (-1)^{\color{black}k+1}\qquad \text{ for }k=0,1,2,\ldots,N.
			\end{equation}
{\color{black}		As before, here we assume that $q_0=0$.}
		\end{itemize}
		
	\end{proposition}
	
	\begin{proof} At first, let us note that the choice $p_0=1$ allows us to extend equation \eqref{pRec}  to work when $k=1$. Next, using the substitution $k\to k-1$ one can rewrite \eqref{pRec} in the following way
		\[
		\frac{p_{k-1}}{p_k}=\cfrac{1} {\mathfrak{a}_{k-1}+\cfrac{p_{k-2}}{p_{k-1}}},
		\]
		which leads to \eqref{nReverseCF}.
		
		To see the second part of the statement, notice that 
		\begin{align*}
			p_{k+1}q_{k} - p_{k}q_{k+1} &= (p_{k}a_{k} + p_{k-1})q_{k} - p_{k}(q_{k}a_{k}+q_{k-1}) \\
			&= p_{k-1}q_{k} - p_{k}q_{k-1}=-(p_kq_{k-1}-p_{k-1}q_k)
		\end{align*} due to \eqref{qRec} and \eqref{pRec}. Thus, \eqref{minus1} follows by induction, since $\color{black}p_1q_0-p_0q_1=-1.$
	\end{proof}
	We will say that a continued fraction of the form \eqref{nConvergents} is {\it palindromic} if its entries $\mathfrak{a}_0$, \dots, $\mathfrak{a}_N$ satisfy the relation
	\[
	\mathfrak{a}_{k}=\mathfrak{a}_{N-k}
	\]
	for all $k\in\{0, 1, 2, \dots, N\}$.
	\begin{remark} \label{parity} Recall that the representation of $q/p$ in the form of \eqref{nCF} when
		$\mathfrak{a}_N\geq2$ (see \cite[Satz 2.2]{Perron}) is unique. However, if $\mathfrak{a}_0=1$ and we still want to get a palindromic continued fraction,  we will need  the last term to be $1$ as well. That is why we can consider a continued fraction \eqref{nCF} where instead of the last term we use the following expression $\mathfrak{a}_{N}=(\mathfrak{a}_{N}-1)+\frac{1}{1}$, which extends the continue fraction by one term and that term is $1$. For instance, if we straightforwardly apply the Euclidean algorithm to $3/4$, we get 
		$$
		\frac{3}{4}=\frac{1}{1+\displaystyle{\frac{1}{3}}},
		$$
		which is not palindromic. However, the extension trick leads to the following continued fraction
		$$
		\frac{3}{4}=\frac{1}{1+\displaystyle{\frac{1}{2+\displaystyle{\frac{1}{1}}}}},$$
		which is palindromic with $\mathfrak{a}_0=\mathfrak{a}_2=1$. It is also worth noting that the trick allows us to have the parity of the number of entries of a continued fraction to be what we want. Say, if it is necessary we can assume that $N$ is even.
	\end{remark}
	Now, we are in the position to formulate and to prove the main statement of this section that can also be found in \cite[Satz 2.4]{Perron} along with its proof.
	\begin{theorem}[The Serret Theorem \cite{Serret}]
		Let $p$ and $q$ be positive integers such that $q<p$. The continued fraction representation of the rational number $q/p$ is palindromic if and only if  either $q^2+1$ or $q^2-1$ is divisible by $p$.
	\end{theorem}
	
	\begin{proof}
		First, we assume that ${q}/{p}$ is palindromic, i.e. 
		
		$$\frac{q}{p} = \cfrac{1}{\mathfrak{a}_{0}+\cfrac{1}{\mathfrak{a}_1+\cfrac{1}{\ddots+\cfrac{1}{\mathfrak{a}_{1}+\cfrac{1}{{\mathfrak{a}_{0}}}}}}}.$$
		Due to the recurrence relation \eqref{nReverseCF}, we also have  $$\frac{p_{ N}}{p_{ N+1}} = \cfrac{1}{\mathfrak{a}_{0}+\cfrac{1}{\mathfrak{a}_1+\cfrac{1}{\ddots+\cfrac{1}{\mathfrak{a}_{1}+\cfrac{1}{{\mathfrak{a}_{0}}}}}}}.$$ So, $q=q_{N+1} = p_{N},$
		where $\{p_k\}$ and $\{q_k\}$ are the sequences generated by this fraction (see~\eqref{nConvergents}).
		From \eqref{minus1} we also know that $$p_{N+1}q_{N} -p_{N}q_{N+1} = (-1)^{\color{black} N+1}.$$  
		Combining these two facts together, we get \begin{equation}\label{minus11}
			pq_{N} - q^2 = (-1)^{\color{black}N+1}.
		\end{equation}
		Thus, $pq_{N} =q^2 + (-1)^{N+1}$. Then, either $q^2+1$ or $q^2-1$  is divisible by $p$.
		
		Next, let us prove the converse is also true. Without loss of generality we can assume that $p$ and $q$ are relatively prime. 
		Consider a continued fraction representation of ${q}/{p}$
		$$\frac{q}{p}= \cfrac{1}{\mathfrak{a}_{0}+\cfrac{1}{\ddots+\cfrac{1}{\mathfrak{a}_{N}}}}.$$
		As before, $\{p_k\}$, $\{q_k\}$ are the sequences generated by this fraction. Let $p$ divide either $q^2+1$ or $q^2-1$. Then there exists an integer $r$ such that {\color{black}{$pr = q^2+(-1)^{N+1}$ or $p_{N+1}r = q_{N+1}^2+(-1)^{N+1}$}}, where we took Remark \ref{parity} into account. Invoking \eqref{minus1}, we can rewrite it as $q_{N+1}(q_{N+1} -p_{N}) = p_{N+1}(r-q_{N})$. This implies that $p_{N+1}$ divides $q_{N+1}(q_{N+1} - p_{N})$. Since $p_{N+1}$ and $q_{N+1}$ are relatively prime,  $p_{N+1}$ divides $q_{N+1} - p_{N}$.
		By our assumption $p_{N+1}>q_{N+1}$. Also, $p_{N+1}=\mathfrak{a}_Np_{N}+p_{N-1}> p_{N}>0$ and, thus, $p_{N+1}>|q_{N+1}-p_N|$.  The latter inequality and the fact that $p_{N+1}$ divides $q_{N+1} - p_{N}$ imply that $q_{N+1} - p_{N}=0$ or $q_{N+1}=p_{N}$.
		Then,
		$$ \cfrac{1}{\mathfrak{a}_{0}+\cfrac{1}{\ddots+\cfrac{1}{\mathfrak{a}_{N}}}}=\frac{q}{p} = \frac{p_{N}}{p_{N+1}} = \cfrac{1}{\mathfrak{a}_{N}+\cfrac{1}{\ddots+\cfrac{1}{\mathfrak{a}_{0}}}}$$ 
		by \eqref{nReverseCF}. Thus, the continued fraction representation of ${q}/{p}$ is palindromic.
	\end{proof}
	
	% In a similar way for $8$ and $13$ we get 	$$
	%	\frac{8}{13}=\frac{1}{1+\displaystyle{\frac{1}{1+\displaystyle{\frac{1}{1+\displaystyle{\frac{1}{1+\displaystyle{\frac{1}{2}}}}}}}}}=\frac{1}{1+\displaystyle{\frac{1}{1+\displaystyle{\frac{1}{1+\displaystyle{\frac{1}{1+\displaystyle{\frac{1}{1+\displaystyle{\frac{1}{1}}}}}}}}}}}.$$
	
	\section{$J$-fractions}
	
	Here we are going to discuss a specific case {\color{black}of polynomial continued fractions which are called the
		Jacobi continued fraction, or shortly $J$-fraction.} More precisely, a continued fraction of the form
	\begin{equation}\label{Jfraction}
		\frac{1}{x-a_0-\displaystyle{\frac{b_0^2}
				{x-a_1-\displaystyle{\frac{b_1^2}{\ddots\displaystyle{-\frac{b_{N-1}^2}{x-a_N}}}}}}},
	\end{equation}
	where $a_k$'s are real and $b_j>0$ for $j=1,\dots N-1$ , is called a $J$-fraction and they can be characterized in the following manner. {\color{black}Note that if we rewrite this fraction as a rational function, i.e. a ratio of two polynomials, then the degree of the numerator is less than the degree of the denominator exactly by 1. }
	
	\begin{theorem}\label{QPJfrac}
		Let $P$ and $Q$ be monic polynomials with real coefficients {\color{black} such that $\deg P-\deg Q=1$}.
		A rational function $m=Q/P$ admits a $J$-fraction representation if and only if
		\begin{itemize}
			\item[(i)] The polynomials $P$ and $Q$ have only real zeros.
			\item[(ii)] The zeroes of $P$ and $Q$ interlace. 
		\end{itemize}
	\end{theorem}
	
	In this statement  if $\mu_{1},\dots,\mu_{N-1}$ are zeros of $Q$ and $\lambda_1,\dots,\lambda_{N}$ are zeros of $P$ then \textit{"interlace"} means $\lambda_1<\mu_{1}<\lambda_2<\ldots<\mu_{N-1}<\lambda_{N}$. 
	\begin{center}
		%	\definecolor{dtsfsf}{rgb}{0.8274509803921568,0.1843137254901961,0.1843137254901961}
		%\definecolor{rvwvcq}{rgb}{0.08235294117647059,0.396078431372549,0.7529411764705882}
		\begin{tikzpicture}[line cap=round,line join=round]
			%,>=triangle 45,x=1cm,y=1cm]
			\begin{axis}[
				x=1cm,y=0.12cm,
				axis lines=middle,
				xmin=-4.3,
				xmax=4.3,
				ymin=-12,
				ymax=12,
				xtick={-4,-3,...,4},
				ytick={-10,-5,...,10},]
				\clip(-4,-14) rectangle (4,14);
				\draw[line width=2pt,dash pattern=on 1pt off 1pt on 1pt off 4pt,%color=rvwvcq,
				smooth,samples=100,domain=-4:4] plot(\x,{((\x)+2)*(\x)*((\x)-2)});
				\draw[line width=2pt,%color=dtsfsf,
				smooth,samples=100,domain=-4:4] plot(\x,{0-((\x)+1)*((\x)-1)*((\x)+2.5)*((\x)-2.5)});
			\end{axis}
			\draw (4.3,-0.5) node{Figure 1. \textit{An example of polynomials with interlacing zeros.}};
		\end{tikzpicture}
	\end{center}
	
	Theorem \ref{QPJfrac} is an immediate consequence of the following well-known statement (for instance, see \cite[Lemma 6.3.9]{RS} where it is proved using the Hermite-Biehler theorem and  as opposed to \cite[Lemma 6.3.9]{RS} we give a less elegant proof but it is more elementary and accessible to a wider audience).  \textcolor{black}{Namely, consecutive applications of Lemma \ref{firstlemma} lead to a slight generalization of \eqref{pRec}, which will soon appear as \eqref{threeP}. }
	
	\begin{lemma}\label{firstlemma}
		Let $P$ and $Q$ be monic real-rooted polynomials such that $\deg{P}=n$ and $\deg{Q}=n-1$ with \textcolor{black}{$n\ge2$}.  The zeros of $P$ and $Q$ interlace if and only if there exists a monic polynomial $R$ of degree $n-2$, a real number $a$, and $b>0$ such that
		\begin{equation}\label{PQR}
			P(x)=(x-a)Q(x)-b R(x).
		\end{equation}
		Moreover, $R$ has only real zeros and they interlace the zeros of $Q$.
	\end{lemma}
	\begin{proof}
		First, let $P(x) = (x-\lambda_1)(x-\lambda_2)\ldots(x-\lambda_n)$ and $Q(x) = (x-\mu_1)(x-\mu_2)\ldots(x-\mu_{n-1})$ have interlacing zeros, i.e. $\lambda_1 < \mu_1 < \lambda_2<\mu_2<\ldots<\mu_{n-1}<\lambda_n$. If $R\equiv0$ then it follows from \eqref{PQR} that $P(x)$ has the same zeros as $(x-a)Q(x)$, which contradicts our assumption. So, $R$ is a nonzero polynomial of degree at most $n-1$.
		
		What we want is to find a real number $a$ such that $R$ in \eqref{PQR}  has  exactly $n-2$ real zeros. 
		Using Vieta's formulas, our guess is that $a = \lambda_1+\lambda_2+\ldots+\lambda_n - (\mu_1+\mu_2+\ldots+\mu_{n-1})$. Since $\lambda_1,\lambda_2,\ldots,\lambda_n$ and $\mu_1,\mu_2,\ldots,\mu_{n-1}$ are all real, then $a$ is also real.
		It is an important observation that $\lambda_1 < a < \lambda_n$, which we are going to use later on. Moreover, by choosing such $a$ we guarantee that $R$ is of degree at most $n-2$.
		
		\begin{center}
			\begin{tikzpicture}
				[line cap=round,line join=round] 
				%domain=-8:2, restrict y to domain=-12:20]%,>=triangle 50]%,x=0.5cm,y=0.25cm]
				\begin{axis}[
					x=1.35cm,y=0.15cm,
					axis lines=middle,
					%ymajorgrids=true,
					%xmajorgrids=true,
					xmin=-4.285936548344514,
					xmax=4.627468452971585,
					ymin=-10.607116891093039,
					ymax=17.306752460886983,
					xtick={-4,-3,...,4},
					ytick={-10,-5,...,15},]
					%\clip(-4,-11.607116891093039) rectangle (4.627468452971585,19.306752460886983);
					\draw[line width=2pt,dash pattern=on 2pt off 3pt,
					smooth,samples=100,domain=-2:4] plot(\x,{((\x)-2.7)*(\x)*((\x)+1.2)*((\x)-3.5)});
					\draw[line width=2pt,smooth,samples=100,domain=-2:4.5] plot(\x,{((\x)+1.6)*((\x)+0.5)*((\x)-1)*((\x)-2.9)*((\x)-3.7)});
					\draw[line width=2pt,dash pattern=on 4pt off 3pt on 1pt off 3pt,
					smooth,samples=100,domain=-1.5:4.5] plot(\x,{((\x)-2.7)*(\x)*((\x)+1.2)*((\x)-3.5)*((\x)-0.5)-((\x)+1.6)*((\x)+0.5)*((\x)-1)*((\x)-2.9)*((\x)-3.7)});
					\draw [line width=1pt] (-2,8)-- (-2,17) ;
					\draw [line width=1pt] (-2,17)-- (-4.2,17);
					\draw [line width=1pt] (-4.2,17)-- (-4.2,8);
					\draw [line width=1pt] (-4.2,8)-- (-2,8);
					\draw (4.5,1) node{$x$};
					\draw (-0.12,16.5) node{$y$};
					\draw [line width=2pt] (-4.05,14.75)-- (-3.3,14.75)node{$\qquad\qquad\quad y=P(x)$};
					\draw [line width=2pt,dash pattern=on 2pt off 3pt] (-4.05,12.2)-- (-3.3,12.2)node{$\qquad\qquad\quad y=Q(x)$};
					\draw [line width=2pt,dash pattern=on 4pt off 3pt on 1pt off 3pt
					] (-4.05,9.65)-- (-3.3,9.65)node{$\qquad\qquad\quad y=R(x)$};
					\begin{scriptsize}
						\draw [fill=black] (-1.6,0) circle (2.5pt);
						\draw [fill=gray] (-1.2,0) circle (2.5pt);
						\draw [fill=gray] (0,0) circle (2.5pt);
						\draw [fill=black] (1,0) circle (2.5pt);
						\draw [fill=gray] (2.7,0) circle (2.5pt);
						\draw [fill=black] (3.7,-1.2811511851395997E-13) circle (2.5pt);
						\draw [fill=gray] (3.5,0) circle (2.5pt);
						\draw [fill=black] (-0.5,0) circle (2.5pt);
						\draw [fill=black] (2.898248089359203,0) circle (2.5pt);
						\draw [fill=white] (-0.7569577746310945,0) circle (2.5pt);
						\draw [fill=white] (3.1898215718450005,0) circle (2.5pt);
						\draw [fill=white] (1.519272954922844,0) circle (2.5pt);
					\end{scriptsize}
				\end{axis}
				\draw  (6,-0.5) node{{Figure 2.} \it An example to illustrate the relation between $P$, $Q$ and $R$.} (6.85,-0.85) node{\it Note that zeros of $P$ and $Q$ (as well as $Q$ and $R$) interlace.};
			\end{tikzpicture}
		\end{center}
		
		{We start by showing that $b>0$ for such an $a$. Note that $R$ is monic. So, from Vieta's formula for the coefficient in front of $x^{n-2}$ in \eqref{PQR} we get
			$$
			b=\sum_{1\le i<j\le n-1} \mu_{i}\mu_j+\left(\sum_{j=1}^{n}\lambda_j-\sum_{i=1}^{n-1}\mu_i\right)\sum_{k=1}^{n-1}\mu_k-\sum_{1\le i<j\le n}\lambda_{i}\lambda_{j}.
			$$
			The right hand side can be rewritten as follows
			$$\sum_{1\le i\le j\le n-1}(\mu_{i}-\lambda_{i})(\lambda_{j+1}-\mu_{j}).
			$$
			From the interlacing of zeros one can see that all the factors in this expression are positive. So $b$ must be positive. }
		
		What is left to do is to show that $R$ changes its sign at least $n-2$ times and as the consequence of the Intermediate Value Theorem we will get the direct statement of this lemma. %It is important to note that $R(x)>0$ when $x>\lambda_{n}$.
		% and $R(x) \neq 0$ when $x<\lambda_1$; namely, $R(x)>0$ if $n$ is even and $R(x)<0$ if $n$ is odd.
		We are going to exhaust all possible values for $a$. So we need to consider the following three cases.
		
		\begin{enumerate}%[label=(\roman*)]
			\item
			$a\in[\mu_i,\lambda_{i+1})$, where $1<i<n-1$.
			
			Analyzing the signs of $P(x)$ and $(x-a)Q(x)$, for $k>i$ we find that $P(\mu_k)>0$ and $(\lambda_{k}-a)Q(\lambda_{k})>0$ when $n-k$ is even and $P(\mu_k)<0$ and $(\lambda_{k}-a)Q(\lambda_{k})<0$ when $n-k$ is odd. Thus, $R(x)=\frac{1}{b}((x-a)Q(x)-P(x))$ has a sign change on $(\lambda_{k},\mu_k)$ for each $k>i$. So what happens at $x=\mu_i$? Either $(x-a)Q(x)$ has the same sign to the left and to the right of this point or it has an additional sign change on $(\mu_i,\lambda_{i+1})$. That is why   for $k<i$ we get $P(\mu_k)>0$ and $(\lambda_{k+1}-a)Q(\lambda_{k+1})>0$ when $n-k$ is even and $P(\mu_k)<0$ and $(\lambda_{k}-a)Q(\lambda_{k+1})<0$ when $n-k$ is odd, implying that $R$ has a sign change on each of $(\mu_k,\lambda_{k+1})$ for $k<i$. Therefore, $R$ changes its sign $n-2$ times at least.
			
			% , where $0 \leq k \leq i-1$, and $(\lambda_m, \mu_m)$, where $i+1 \leq m \leq n$. Moreover, and $R(\mu_k)$ and $R(\lambda_{k+1})$ differ in sign and $R(\lambda_m)$ and $R(\mu_m)$ differ in sign. Clearly, $P(x)$ and $(x-a)Q(x)$ do not intersect on the intervals $(\mu_0, \lambda_1)$ and $(\lambda_n, \mu_n)$ because $bR(x)>0$ for all $x> \lambda_n$ and $bR(x)>0$ or $bR(x)<0$ when $x<\lambda_1$, depending on whether $n$ is even or odd. So $(x-a)Q(x)-P(x) = bR(x)=0$ on these $n-2$ intervals due to the Intermediate Value Theorem. We have now found our $n-2$ zeros. 

			\item
			$a = \lambda_i$ for some $i$ between 2 and $n-1$.
			
			First of all notice that $R(\lambda_i)=0$ in this case. Moreover, for $k>i$ we still get a sign change for $R$ on each of the intervals $(\lambda_{k},\mu_k)$. However, $R$ has a sign change on each of $(\mu_k,\lambda_{k+1})$ now for $k<i-1$. So all together we have 1 zero and $n-3$ sign changes at the points distinct from that zero.

			\item 
			$a \in (\lambda_i,\mu_{i})$, where $i=1,\dots,n-1$
			
			By the argument similar to (1), $R$ changes sign $n-2$ times. 
			
		\end{enumerate}
		In any case $R$ has $n-2$ real zeros, i.e. it is exactly a polynomial of degree $n-2$, and from the construction of the argument above it is not hard to see that the zeros of $R$ interlace with the zeros of $Q$.
		
		To prove the converse statement assume that $P(x)=(x-a)Q(x)-b R(x)$, where $b$ is positive, $a$ is real, and $R$ has real zeros that interlace the zeros of $Q(x)$. Note that  $\deg P=n$.
		
		Let $R(x) = (x-r_1)(x-r_2) \dots (x-r_{n-2})$ and $Q(x) = (x-\mu_1)(x-\mu_2) \dots (x-\mu_{n-1})$, where $\mu_1<r_1<\mu_2< \dots < r_{n-2} < \mu_{n-1}$.
		
		Note that the signs of $R(\mu_i) = {-P(\mu_i)}/{b}$ and $R(\mu_{i+1}) = {-P(\mu_{i+1})}/{b}$ are opposite for $i=1, \dots, n-2$, since $R$ has a simple zero on each interval $(\mu_i,\mu_{i+1})$.
		By the Intermediate Value Theorem, $P$ must have a zero as well in each of the intervals $(\mu_i,\mu_{i+1})$, that is, we have found $n-2$ real zeros of $P$. 
		
		We want to show that $P$ has two more zeros in the intervals $(-\infty,\mu_1)$ and $(\mu_{n-1}, \infty)$.
		Note that  $P(\mu_{n-1})=(\mu_{n-1}-a)Q(\mu_{n-1})-bR(\mu_{n-1})=-bR(\mu_{n-1})<0$ and $P$ is a monic polynomial. Hence, $P(x)$ is going to be positive for sufficiently large $x$, proving that it has a zero on $(\mu_{n-1}, \infty)$ as a consequence of the Intermediate Value Theorem. By the similar argument, since degrees of $P$ and $R$ differ by 2 and they suppose to have the same sign in a neighborhood of $-\infty$. However,  $P(\mu_{1})=-bR(\mu_{n})$ which proves that there should be a zero of $P$ in $(-\infty,\mu_1)$.
		
		%$$\frac{d}{dx} (x-a)Q(x) = \frac{Q(x)(x-a)}{x-\mu_1}+ \frac{Q(x)(x-a)}{x-\mu_2}+\ldots+\frac{Q(x)(x-a)}{x-\mu_{n-1}}+Q(x)$$ and 
		%
		%$$\frac{d}{dx} R(x) = \frac{R(x)}{x-r_1}+\frac{R(x)}{x-r_2}+\ldots+\frac{R(x)}{x-r_n}$$.
		%
		%Let $\xi = \max\{\mu_{n-1}, a\}$. If $x> \xi$, then $\frac{Q(x)(x-a)}{x-\mu_p} > \frac{R(x)}{x-r_{p}}$ for $p =1, \dots, n-2$. Thus, $\frac{d}{dx}Q(x)(x-a) > \frac{d}{dx}R(x)$. Since $Q(\xi)(\xi-a) < R(\xi)$, $Q(x)(x-a)$ will intersect $R(x)$ at a sufficiently large $ x $; thus, $P(x)$ has a zero in the interval $(\mu_{n-1}, \infty)$.
		%
		%Now let $\xi =$ min[$a, \mu_1$]. Now $\left| \frac{d}{dx}Q(x)(x-a) \right| > \left| \frac{d}{dx} R(x) \right|$ for $x< \xi$ because $\left| \frac{Q(x)(x-a)}{x-\mu_p} \right|> \left| \frac{R(x)}{x-r_p} \right| $ for all $p<n-1$. It's easy to see that $\frac{Q(x)(x-a)}{x-\mu_p}$ and $\frac{R(x)}{x-r_p}$ are positive and $Q(x)(x-a)$, $P(x)$ are negative for odd $n$; $\frac{Q(x)(x-a)}{x-\mu_p}$ and $\frac{R(x)}{x-r_p}$ are negative and $Q(x)(x-a)$, $P(x)$ are positive for even $n$. Since $Q(\xi)(x-a) < P(\xi)$ for even $n$ and $Q(\xi)(x-a) > P(\xi)$ for odd $n$, $Q(x)(x-a)$ will intersect $P(x)$ on the interval $(-\infty, \mu_1)$. 
		
		We have found $n$ distinct real zeros. Moreover, the argument above implies that the zeros of $P$ interlace the zeros of $Q$.
	\end{proof}
	
	\section{A Characterization of Palindromic $J$-fractions}
	%\cite{perron1913lehre} gave a similar proof for integer continued fractions, which we will emulate. 
	In this section we consider  palindromic $J$-fractions and give necessary and sufficient conditions for a rational function to have a palindromic $J$-fraction. \textcolor{black}{This characterization reveals another face of the Serret theorem and at the same time it demonstrates that Serret's result goes beyond the case of positive integers.}
	
	Let us start by saying that a $J$-fraction 
	\begin{equation}\label{Jfrac}
		\frac{1}{x-a_0-\displaystyle{\frac{b_0^2}
				{x-a_1-\displaystyle{\frac{b_1^2}{\ddots\displaystyle{-\frac{b_{N-1}^2}{x-a_N}}}}}}}
	\end{equation}
	is called {\it palindromic} if $a_i=a_{N-i}$ for $i\in\{0,\dots,N\}$ and $b_j=b_{N-j-1}$ for $j\in\{0,\dots,N-1\}$. 
	For a $J$-fraction we can define sequences of polynomials in exactly the same manner as we did before for numerical sequences. Namely, define polynomials $P_k$'s and $Q_k$'s by the following recurrence relations: 
	\begin{equation}\label{threeP}
		P_{k+1}(x) = P_{k}(x)(x-a_{k}) - b^2_{k-1}P_{k-1}(x), \quad k=0,1,2,\dots, N
	\end{equation}
	and \begin{equation}\label{threeQ}
		Q_{k+1} (x)= Q_{k}(x)(x-a_{k}) - b^2_{k-1}Q_{k-1}(x),\quad k=0,1,2,\dots, N,
	\end{equation}
	where $P_{-1}(x) = 0$, $P_0 (x)= 1$, $Q_{-1} (x)= -1$, $Q_0(x) =0$ and $b_{-1}=1$. 
	Similarly to the numeric continued fraction case, the sequences $P_k$ and $Q_k$ possess the following properties. 
	\begin{proposition}\label{lemma1} Let  $\{P_k\}$ and $\{Q_k\}$ be sequences of polynomials constructed by the  $J$-fraction \eqref{Jfrac}. Then 
		\begin{itemize}
			\item[(i)] $\frac{P_{N}}{P_{N+1}}$ has the following continued fraction fraction representation \begin{equation}\label{JfracP}
				\frac{P_{N}(x)}{P_{N+1}(x)}=\frac{1}{x-a_{N}-\displaystyle{\frac{b_{N-1}^2}
						{x-a_{N-1}-\displaystyle{\frac{b_{N-2}^2}{\ddots\displaystyle{-\frac{b_{0}^2}{x-a_0}}}}}}}.
			\end{equation}
			\item[(ii)] The following Liouville-Ostrogradski formula 
			\begin{equation}\label{bsquare}
				\det\begin{pmatrix}
					P_{k+1}&P_{k}\\Q_{k+1}& Q_{k}
				\end{pmatrix} = -b_0^2b_1^2\dots b^2_{\color{black}k-1}
			\end{equation}
			holds true for $\color{black}k=1,2,\dots, N$.
		\end{itemize}
	\end{proposition}

	\begin{proof} Formula \eqref{JfracP} immediately follows from the relation
		\begin{equation}\label{PkPk-1}\color{black}
			\frac{P_{k}(x)}{P_{k+1}(x)}=\cfrac{1}{x-a_k-b^2_{k-1}{\cfrac{P_{k-1}(x)}{P_{k}(x)}}}
		\end{equation}
		which is  just another form of \eqref{threeP}.
		%Continuing the process,
		%combining it with the Euclidean algorithm  
		%we get the desired continued fraction.
		
		As for the second part, notice that relations \eqref{threeP} and \eqref{threeQ} imply
		\begin{align*}\color{black}
			P_{k+1}(x)Q_{k}(x) - P_{k}(x)Q_{k+1}(x) &= Q_{k}(x)((x-a_{k})P_{k}(x) - b^2_{k-1}P_{k-1}(x))\\&-((x-a_{k})Q_{k}(x) - b^2_{k-1}Q_{k-1}(x))P_{k}(x) 
			& \\
			&= b^2_{k-1}(P_{k}(x)Q_{k-1}(x) - P_{k-1}(x)Q_{k}(x)). 
		\end{align*}
		Therefore,  \eqref{bsquare} follows by induction, since $$\color{black} P_{2}(x)Q_{1}(x)-P_{1}(x)Q_{2}(x) =((x-a_{0})(x-a_{1})-b_0^2)\cdot1- (x-a_{0})(x-a_{1}) =-b_0^2.$$ 
		This concludes the proof.
	\end{proof}
	
	\begin{theorem}\label{theorem1} Let  $P$ and $Q$ be real polynomials and  ${\deg }P = {\deg }Q+1$. Then ${Q}/{P}$ has a palindromic $J$-fraction representation if and only if 
		\begin{itemize}
			\item[(i)] $P$ and $Q$ have real zeros,
			\item[(ii)] the zeros of $P$ and $Q$ interlace,
			\item[(iii)]  $Q^2-b_0^2b_1^2 \dots b_{N-1}^2$ is divisible by $P$.
			
		\end{itemize}
	\end{theorem}
	\begin{proof}
		By Theorem~\ref{QPJfrac} (i) and (ii) are equivalent to the fact that $Q/P$ has a $J$-fraction representation.
		
		Now, suppose (iii) holds, i.e. $P$ divides $Q^2 - b_0^2b_1^2 \dots b_{N-1}^2$ and combining that with $\frac{Q_{N+1}}{P_{N+1}} = \frac{Q}{P}$, we get that there exists a polynomial $ R$  such that $PR = Q^2 - b_0^2b_1^2 \dots b_{\color{black}N-1}^2$ or that 
		$$ Q_{N+1}^2 - b_0^2b_1^2 \dots b_{N-1}^2 =P_{N+1}R.$$
		From Proposition~\ref{lemma1}, we know that 
		$$P_{N}Q_{N+1} - b_0^2b_1^2 \dots b_{N-1}^2 = P_{N+1}Q_{N}.$$ 
		Subtraction gives $$Q_{N+1}(Q_{N+1}-P_{N}) = P_{N+1}(R- Q_{N})$$ which implies that $P_{N+1}$ divides $Q_{N+1}(Q_{N+1} - P_{N})$. However, $P_{N+1}$ and  $Q_{N+1}$ have no non-constant common divisors. Thus, $P_{N+1}$ divides $Q_{N+1} -P_{N}$.
		
		Moreover, $\deg P_{N+1} > {\deg }Q_{N+1}$ and ${\deg }P_{N+1} > {\deg }P_{N}$. 
		So we conclude that ${\deg }P_{N+1} > {\deg }(Q_{N+1}-P_{N})$ and, hence, $Q_{N+1} - P_{N}=0$. Then, by \eqref{threeP} and the previous argument we get \textcolor{black}{
			\begin{equation*}
				\frac{Q_{N+1}(x)}{P_{N+1}(x)} =    \frac{P_N(x)}{P_{N}(x)(x-a_{N}) - b^2_{N-1}P_{N-1}(x)}=\frac{1}{(x-a_{N}) - \displaystyle{\frac{b^2_{N-1}}{\;\;\displaystyle{\frac{P_{N}(x)}{P_{N-1}(x)}\;\;}}}}.
		\end{equation*}}
		%\begin{equation}
		%    \label{palfrac}
		%    =\frac{1}{x-a_n-\displaystyle{\frac{b_{n-1}^2}
		%{x-a_{n-1}-\displaystyle{\frac{b_{n-1}^2}{\ddots\displaystyle{-\frac{b_{0}^2}{x-a_0}}}}}}}= \frac{1}{x-a_0-\displaystyle{\frac{b_0^2}
		%{x-a_1-\displaystyle{\frac{b_1^2}{\ddots\displaystyle{-\frac{b_{n-1}^2}{x-a_n}}}}}}}= \frac{P_{n-1}}{P_n}.
		%\end{equation}
		
		{Thus, by inductively continuing and then  comparing it with \eqref{PkPk-1}, one can see that the continued fraction representation of ${Q}/{P}$ is palindromic.
			
			On the other hand, if we assume that ${Q}/{P}$ can be expressed as a palindromic continued fraction and notice that $\deg P = {\deg }Q+1$ and {\color{black}$  P_{N+1}(x) = Q_{N+1}(x)(x-a_{N}) - b^2_{N-1}P_{N-1}(x)$},  it is an immediate consequence from Lemma \ref{firstlemma} that the zeros of $P$ and $Q$ interlace. 
			
			We  also know that ${Q} = {P_{N}}$. By Lemma \ref{lemma1}, {\color{black} $P_{N}Q_{N+1}- P_{N+1}Q_{N} =-(P_{N+1}Q_{N}-P_{N}Q_{N+1}) = -(-b_0^2 \dots b^2_{N-1})=b_0^2 \dots b^2_{N-1}$. That is}, we have $P_{N+1}Q_{N} = Q^2_{N+1} - b_0^2 \dots b^2_{N-1}$.
			Thus, $Q^2 - b_0^2 \dots b^2_{N-1}$ is divisible by  $P_{N+1}=P$.}
	\end{proof}
	
	%\section{Chebyshev Polynomials}
	
	To demonstrate this theorem, let us recall that {\it the Chebyshev polynomials of the first kind}  are defined by the  three-term recurrence relation
	\begin{equation*}
		T_{n+1}(x)=2xT_n(x)-T_{n-1}(x),
	\end{equation*}
	where $T_0(x)=1$ and $T_1(x)=x$. Also, {\it the Chebyshev polynomials of the second kind} are given by the recurrence relation
	$$U_{n+1}(x)=2xU_{n}(x)-U_{n-1}(x),$$
	where $U_{0}(x)=1$  and $U_{1}(x)=2x$.
	
	\begin{center}
		
		\begin{tikzpicture}[line cap=round,line join=round]
			%,>=triangle 45,x=1cm,y=1cm]
			\begin{axis}[
				x=3cm,y=1cm,
				axis lines=middle,
				xmin=-1.74,
				xmax=1.3,
				ymin=-1.8,
				ymax=2.1,
				xtick={-1.5,-1,...,1},
				ytick={-1.5,-1,...,1.5},]
				\draw (1.2,0.2) node{$x$};
				\draw (-0.1,1.9) node{$y$};
				%\clip(-1.7,-2) rectangle (1.3,2);
				\draw[line width=2pt,dash pattern=on 3pt off 3pt on 1pt off 4pt,%color=sexdts,
				smooth,samples=100,domain=-1.7:1.3] plot(\x,{16*(\x)^(5)-20*(\x)^(3)+5*(\x)});
				\draw[line width=2pt,smooth,samples=100,domain=-1.7:1.3] plot(\x,{16*(\x)^(4)-12*(\x)^(2)+1});
				\draw [line width=1pt] (-1.71,1.5)-- (-1.71,0.5);
				\draw [line width=1pt] (-1.71,0.5)-- (-0.94,0.5);
				\draw [line width=1pt] (-.94,0.5)-- (-.94,1.5);
				\draw [line width=1pt] (-.94,1.5)-- (-1.71,1.5);
				\draw [line width=2pt] (-1.68,1.2)-- (-1.5,1.2)node[right]{$y=U_4(x)$};
				\draw [line width=2pt,dash pattern=on 3pt off 3pt on 1pt off 4pt%,color=wwzzqq
				] (-1.68,0.75)-- (-1.5,0.75) node[right]{$y=T_5(x)$};
				\begin{scriptsize}
					\draw [fill=gray] (-0.9510565162951536,0) circle (2.5pt);
					\draw [fill=gray] (-0.5877852522924729,0) circle (2.5pt);
					\draw [fill=gray] (0,0) circle (2.5pt);
					\draw [fill=gray] (0.5877852522924737,0) circle (2.5pt);
					\draw [fill=gray] (0.9510565162951533,0) circle (2.5pt);
					\draw [fill=black] (-0.8090169943749475,0) circle (2.5pt);
					\draw [fill=black] (-0.3090169943749474,0) circle (2.5pt);
					\draw [fill=black] (0.30901699437494734,0) circle (2.5pt);
					\draw [fill=black] (0.8090169943749476,0) circle (2.5pt);
				\end{scriptsize}
			\end{axis}
			\draw  (4.5,-0.5) node{{Figure 3.} \it Plots of the Chebyshev polynomials $T_5$ and $U_4$. Note that} (0.8,-0.85)  node[right]{\it not only zeros of $T_5$ and $U_4$  interlace but also the zeros  of } (0.8,-1.2) node[right]{\it $U_4$ are at the location of the local extrema of $T_5$.};
		\end{tikzpicture}
	\end{center}
	The reason we bring up the Chebyshev polynomials here is because they satisfy the Pell-Abel equation 
	\begin{equation}\label{Pell}
		T_n(x)^2 - (x^2-1)U_{n-1}(x)^2 =1,
	\end{equation}
	which looks similar to the third condition in Theorem \ref{theorem1}. To see the validity of this observation, note that from the recurrence relations it follows that on the interval $[-1,1]$ the Chebyshev polynomials can be written as  
	\begin{equation}\label{trigChebyshev}
		T_n(\cos{\theta}) = \cos{n\theta}, \quad U_{n-1}(\cos{\theta}) = \frac{\sin{n\theta}}{\sin{\theta}}
	\end{equation}
	and, thus, \eqref{Pell} is nothing else but the trigonometric identity
	\[
	\cos^2(n\theta)+\sin^2(n\theta)=1.
	\]
	At the same time, the relations in \eqref{trigChebyshev}  allow us to find the zeroes of the Chebyshev polynomials and we then see that their zeros interlace. Therefore, Theorem \ref{theorem1} implies  that the continued fraction representation of $\frac{T_n(x)}{(x^2-1)U_{n-1}(x)}$ must be palindromic. Actually, in this particular case, we can find this palindromic continued fraction explicitly. 
	
	\begin{theorem}\label{PalindromicChebyshev}
		The continued fraction expansion of $\frac{T_n(x)}{(x^2-1)U_{n-1}(x)}$ is palindromic and of the form
		\begin{equation}\label{cheb}
			\frac{1}{x-\displaystyle{\frac{\frac{1}{2}}
					{x-\displaystyle{\frac{\frac{1}{4}}{\ddots-\displaystyle{\frac{\frac{1}{4}}{x-\displaystyle{\frac{\frac{1}{2}}{x}}}}}}}}}.
		\end{equation}
		
	\end{theorem}
	
	\begin{proof}
		It is well known that 
		\[
		T^\prime_n(x)=nU_{n-1}(x), \quad  (1-x^2)T'_{n}(x) = {-nxT_n(x)+nT_{n-1}(x)}
		\]
		and these formulas can be easily derived from  \eqref{trigChebyshev}.
		We apply the above relations to get the following 
		$$\frac{T_n(x)}{(x^2-1)U_{n-1}(x)} =\frac{nT_n(x)}{(x^2-1)T'_n(x)} = \frac{nT_n(x)}{nxT_n(x)-nT_{n-1}(x)}=\frac{1}{x-\displaystyle{\frac{1/2}{\;\;\displaystyle{\frac{T_{\color{black}n}(x)}{2T_{\color{black}n-1}(x)}}\;\;}}} .$$
		Next, we utilize the three term recurrence relation for the Chebyshev polynomials of the first kind in order to obtain the relation
		$$\color{black}\frac{T_{k}(x)}{2T_{k-1}(x)}=x-\frac{T_{k-2}(x)}{2T_{k-1}(x)}=x-\frac{1/4}{\;\;\displaystyle{\frac{T_{k-1}(x)}{2T_{k-2}(x)}}\;\;},$$
		and inductively apply it until one  gets
		$$\frac{T_{2}(x)}{2T_{1}(x)}=x-\frac{1/2}{x}.$$
		Thus, $\frac{T_n(x)}{(x^2-1)U_{n-1}(x)}$ is of the form \eqref{cheb}.
	\end{proof}
	\begin{remark} In \cite{SYu} a characterization of the polynomials satisfying a generalization of the Pell-Abel equation \eqref{Pell} is given. {\color{black} The} construction described in Theorem \ref{PalindromicChebyshev} can be extended to those polynomials. 
	\end{remark}

	\section{Jacobi matrices}
	
	Introducing the polynomials
	\[
	\widehat{P_k}(x)=\frac{1}{b_0b_1\dots b_{k-1}}P_k(x), \quad k=1,2,\dots, N+1
	\]
	we can rewrite 	\eqref{threeP} in the following manner
	\begin{equation}\label{3termrecsym}
		x\widehat{P}_{k}(x) =  b_{k} \widehat{P}_{k+1}(x) +a_{k}\widehat{P}_{k}(x)+ b_{k-1}\widehat{P}_{k-1}(x), \quad k=0,1,2,\dots, N,
	\end{equation}
	where we set $b_N=1$ for convenience {\color{black} and as before $b_{-1}=1$. The initial conditions then take the form $\widehat{P}_{-1}=0$, $\widehat{P}_0=1$}. 
	Next, recall that a finite Jacobi matrix is a symmetric tridiagonal matrix of the following form
	\[
	H_{[0,N]}=\left(%
	\begin{array}{cccc}
		a_0 & b_0    &  &  \\
		b_0 & a_1    & \ddots &  \\
		& \ddots & \ddots & b_{N-1} \\
		&        & b_{N-1}& a_{N}     \\
	\end{array}%
	\right),
	\]
	where $a_j$ are real numbers and $b_j$ are positive numbers. With the help of this matrix constructed from the coefficients of the recurrence relation, \eqref{3termrecsym} takes the form
	\begin{equation}\label{threetermJ}
		x\bar{P}(x)=H_{[0,N]}\bar{P}(x)+\begin{pmatrix}
			0\\ \vdots\\0\\\widehat{P}_{N+1}(x)
		\end{pmatrix},
	\end{equation}
	where $\bar{P}(x)=(
	\widehat{P}_0(x),
	\widehat{P}_1(x),
	\dots,
	\widehat{P}_N(x))^\top $ is a vector consisting of polynomials $\widehat{P}_j$'s. Similarly, let us introduce the polynomials
	\[
	\widehat{Q}_k(x)=\frac{1}{b_0b_1\dots b_{k-1}}Q_k(x), \quad k=1,2,\dots, N+1,
	\]	
	{\color{black}$\widehat{Q}_{-1}=-1$, $\widehat{Q}_0=0$,} and a truncation of the Jacobi matrix $H_{[0,N]}$
	\[
	H_{[1,N]}=\left(%
	\begin{array}{cccc}
		a_1 & b_1    &  &  \\
		b_1 & a_2    & \ddots &  \\
		& \ddots & \ddots & b_{N-1} \\
		&        & b_{N-1}& a_{N}     \\
	\end{array}%
	\right),
	\]
	which is also a Jacobi matrix. Then, one can represent \eqref{threeQ} in the matrix form
	\begin{equation*}%\label{threetermJQ}
		x\bar{Q}(x)=H_{[1,N]}\bar{Q}(x)+\begin{pmatrix}
			0\\ \vdots\\0\\\widehat{Q}_{N+1}(x)
		\end{pmatrix},
	\end{equation*}
	where $\bar{Q}(x)=(
	\widehat{Q}_1(x),
	\widehat{Q}_2(x),
	\dots,
	\widehat{Q}_N(x))^\top $ is a vector consisting of polynomials $\widehat{Q}_j$'s.
	These matrix representations of the three-term recurrence relations allow us to derive some determinant formulas for the polynomials $\widehat{P}_{N+1}$ and $\widehat{Q}_{N+1}$.
	\begin{proposition}[\cite{B} p. 542, \cite{GS}]\label{detFormula} We have that
		\[
		\begin{split}
			\widehat{P}_{N+1}(x)=\frac{1}{b_0b_1\dots b_{N-1}}\det(xI-H_{[0,N]}), \\
			\widehat{Q}_{N+1}(x)=\frac{1}{b_0b_1\dots b_{N-1}}\det(xI-H_{[1,N]}),
		\end{split}
		\]
		where $I$ is the identity matrix of the proper size.
	\end{proposition}	
	\begin{proof}
		First, let us show that $H_{[0,N]}$ has $N+1$ distinct eigenvalues. To this end, consider $H_{[0,N]}-\lambda I$ where $\lambda$ is an eigenvalue of $H_{[0,N]}$. Observe that it has rank at most $N$. Indeed, due to \cite[Theorem 5S, p. 297]{Strang} we have that
		\begin{equation}\label{diag}
			H_{[0,N]}={\color{black}\mathcal{U}}\begin{pmatrix}
				\lambda&&&\\
				&\star&&\\
				&&\ddots&\\
				&&&\star
			\end{pmatrix}\mathcal{U}^\top \qquad\text{ or }\qquad H_{[0,N]}-\lambda I=\mathcal{U}\begin{pmatrix}
				0&&&\\
				&\ast&&\\
				&&\ddots&\\
				&&&\ast
			\end{pmatrix}\mathcal{U}^\top ,
		\end{equation} 
		where $\color{black}\mathcal{U}$ is an orthogonal matrix, $\star$ is in place of an eigenvalue, and $\ast$ stands for a difference of eigenvalues and may be 0. It is clear that the diagonal matrix on the right is of the rank at most $N$ and so is $H_{[0,N]}-\lambda I$ as similar to it. 
		
		On the other hand, the following submatrix of $H_{[0,N]}-\lambda I$ is an upper triangular matrix with non-zero entries.\textcolor{black}{
			$$\begin{pmatrix}
				b_0&a_1-\lambda&b_1&0&\dots&0\\
				0&b_1&a_2-\lambda&b_2&\ddots&\vdots\\
				\vdots&\ddots&\ddots&\ddots&\ddots&0\\
				\vdots&&\ddots&\ddots&\ddots&b_{N-1}\\
				\vdots&&&\ddots&\ddots&a_{N-1}-\lambda\\
				0&\dots&\dots&\dots&0&b_{N-1}\\
			\end{pmatrix}.$$  }
		This means that this matrix is non-degenerate and of rank $N$. Therefore, the rank of $H_{[0,N]}-\lambda I$ must be exactly $N$. That is, none of the $\ast$'s in \eqref{diag} are actually 0 which implies that the arithmetic multiplicity of $\lambda$ is 1. From the arbitrary choice of the eigenvalue $\lambda$ in the beginning we deduce that  $H_{[0,N]}$ has $N+1$ distinct eigenvalues as we claimed before.
		
		Next, assume that $\lambda$ is a zero of $\widehat{P}_{N+1}$. Then, it follows from \eqref{threetermJ}  that  $J\cdot \bar{P}(\lambda)=\lambda \bar{P}(\lambda)$ and $\bar{P}(\lambda)\ne\bar{0}$ since $\widehat{P}_0(x)=1$. Thus, $\lambda$ is an eigenvalue of $H_{[0,N]}$ with $\bar{P}(\lambda)$ being the corresponding eigenvector. As a result, $\det(\lambda I-H_{[0,N]})=0$ and so the determinant formula for $\widehat{P}_{N+1}$ follows from this observation and the fact that the eigenvalues of $H_{[0,N]}$ are distinct. Analogously, one can prove the determinant formula for $\widehat{Q}_{N+1}$.
	\end{proof}
	%To conclude this section, consider the standard inner product in $\dC^{N+1}$
	%\[
	%(f,g)=\sum_{i=0}^{N}f_i\overline{g}_i,\quad f,g\in\dC^{N+1}.
	%\]
	%and define the $m$-function of the Jacobi matrix $H_{[0,N]}$ \cite{GS}
	%\[
	%m(x)=\left((x-H)^{-1}e_0,e_0\right)
	%\]
	%where $e=(1,0,\dots,0)^\top\in\dC^{N+1}$. Next, it is not so hard to see that
	%$$
	%m(x)=\frac{\det(x-H_{[1,N]})}{\det(x-H_{[0,N]})},$$
	%where $Q_N(x)=\det(x-H_{[1,N]})$ and $P_N(x)=\det(x-H_{[0,N]})$.
	%Then one can see that the $m$-function admits the following continued fraction representation\ldots

	\section{Perfect quantum state transfer}
	
	The study of quantum state transfer was initiated by S. Bose \cite{Bose03},  who  considered  a  $1D$  chain  of $N$
	qubits \textcolor{black}{whose interactions are expressed by a time-independent Hamiltonian. 
		The motivation for such a study is that in a quantum computer one needs to entangle distant qubits while the interactions between them are typically local. Thus, the idea is for a chain of an arbitrary size to find Hamiltonians that allow us to transport a quantum state from one end of the chain to the other.} We say the transport
	of quantum state from one location to another is perfect
	if it is realized with probability 1, that is, without  dissipation. 
	A few cases, when perfect transmission can be achieved, \textcolor{black}{have been found
		in some $1D$ chains for which only nearest neighbors are assumed to be interacting} (for instance, see 
	\cite{Kay10} and the references therein). In these cases, the
	probability for the transfer of a single spin excitation from one
	end of the chain to the other is found to be 1 for certain
	times. These models have the advantage that the perfect transfer
	can be done without the need for active control. 
	
	For such $1D$  chains the Hamiltonian is given by a Jacobi matrix $H=H_{[0,N]}$ and the mathematical aspect of the problem of perfect quantum state transfer is that perfect transmission can be achieved if there exist a real number $\varphi$ and a positive number $T$ such that
	\begin{equation}\label{PSTcond}
		e^{i\varphi}e^{iTH_{[0,N]}}e_0=e_N,
	\end{equation}
	where $e_0=(1, 0,\dots,0)^\top$ and $e_N=(0, \dots,0,1)^\top$. 
	
	Since Jacobi matrices appear in this context, one can invoke the previously developed machinery since a Jacobi matrix $H=H_{[0,N]}$ defines a sequence of polynomials $\widehat{P}_0$,  $\widehat{P}_1$, \dots,  $\widehat{P}_{N+1}$ by formula \eqref{3termrecsym}. Moreover, this relation allows us to obtain a characterization of Jacobi matrices realizing perfect state transfer.
	\begin{theorem}[\cite{Kay10}, \cite{VZh12}]
		Given $H_{[0,N]}$ the condition \eqref{PSTcond} holds true for some $\varphi\in\dR$ and $T>0$ if and only if the following two conditions are satisfied
		\begin{enumerate}
			\item[(i)] We have that
			\begin{equation}\label{PST1}
				e^{i(T\lambda_k+\varphi)}=(-1)^{N+k}, \quad k=0,1,2,\dots,N,
			\end{equation}
			where $\lambda_k$'s are the eigenvalue of $H_{[0,N]}$ ordered as follows
			\[
			\lambda_0<\lambda_1<\dots<\lambda_N.
			\]
			\item[(ii)] We have that 
			\begin{equation}\label{PST2}
				\widehat{P}_N(\lambda_k)=(-1)^{N+k}, \quad k=0,1,2,\dots,N,
			\end{equation}
			where $\widehat{P}_N$ is the polynomial defined by \eqref{3termrecsym} using the entries of the given Jacobi matrix.
		\end{enumerate}
	\end{theorem}
	\begin{proof}
		Using \eqref{3termrecsym}, we can construct $\widehat{P}_0$,  $\widehat{P}_1$, \dots,  $\widehat{P}_{N+1}$ and so the eigenvectors. Indeed, due to \eqref{threetermJ}, the vectors
		\[
		\bar{P}(\lambda_0), \bar{P}(\lambda_1), \dots, \bar{P}(\lambda_N)
		\]
		are the $N+1$ linearly independent vectors of  $H_{[0,N]}$ since a Jacobi matrix has distinct eigenvalues.
		Next, let us consider the matrix {\color{black}$\mathcal{P}=\left(\frac{\bar{P}(\lambda_0)}{\|\bar{P}(\lambda_0)\|}\;\dots \;\frac{\bar{P}(\lambda_{\color{black}N})}{\|\bar{P}(\lambda_N)\|}\right)$. This matrix $\mathcal{P}$ is an orthogonal matrix, that is $\mathcal{P}^{\top}\mathcal{P}=\mathcal{P}\mathcal{P}^{\top}=I$, since all the eigenvalues are distinct and the matrix being real symmetric implies that eigenvectors corresponding to two distinct eigenvalues are orthogonal 
			(see e.g. \cite[Theorem 5S, p. 297]{Strang} ). Besides, we also learn from \cite[Theorem 5S, p. 297]{Strang} that $H_{[0,N]}=\mathcal{P}D\mathcal{P}^{\top}$, where $D=\text{diag}\{\lambda_0,\dots,\lambda_N\}$ is diagonal, and $\lambda_0$, \dots, $\lambda_N$ are the eigenvalues of $H_{[0,N]}$. Next, we want to find $e^{iTH_{[0,N]}}$, which is an easy task since $H_{[0,N]}$  is diagonalizable \cite[Section 5.4]{Strang}.  Therefore, we have
			$$
			e^{iTJ}=\mathcal{P}\begin{pmatrix}
				e^{iT\lambda_0}&&\\
				&\ddots&\\
				&&e^{iT\lambda_N}
			\end{pmatrix}\mathcal{P}^\top .$$
			Then the condition \eqref{PSTcond} reduces to the following condition
			\begin{equation*}
				e^{i\varphi}
				\mathcal{P}\begin{pmatrix}
					e^{iT\lambda_0}&&\\
					&\ddots&\\
					&&e^{iT\lambda_N}
				\end{pmatrix}\mathcal{P}^\top e_0=e_N     
			\end{equation*}
			or
			\begin{equation}\label{Ptransfer}
				e^{i\varphi}\begin{pmatrix}
					e^{iT\lambda_0}&&\\
					&\ddots&\\
					&&e^{iT\lambda_N}
				\end{pmatrix}\mathcal{P}^\top e_0=\mathcal{P}^\top e_N.
			\end{equation}
			Moreover, $\mathcal{P}^\top e_0=\begin{pmatrix}
				\frac{\widehat{P}_0(\lambda_0)}{\|\bar{P}(\lambda_0)\|}\\
				\frac{\widehat{P}_0(\lambda_1)}{\|\bar{P}(\lambda_1)\|}\\
				\vdots\\
				\frac{\widehat{P}_0(\lambda_N)}{\|\bar{P}(\lambda_N)\|}
			\end{pmatrix}$ and $\mathcal{P}^\top e_N=\begin{pmatrix}
				\frac{\widehat{P}_N(\lambda_0)}{\|\bar{P}(\lambda_0)\|}\\
				\frac{\widehat{P}_N(\lambda_1)}{\|\bar{P}(\lambda_1)\|}\\
				\vdots\\
				\frac{\widehat{P}_N(\lambda_N)}{\|\bar{P}(\lambda_N)\|}
			\end{pmatrix}$}. Hence, \eqref{Ptransfer} is equivalent to
		\begin{equation}
			e^{i\varphi}\begin{pmatrix}
				\frac{e^{iT\lambda_0}	\widehat{P}_0(\lambda_0)}{\|\bar{P}(\lambda_0)\|}\\
				\frac{e^{iT\lambda_1}	\widehat{P}_0(\lambda_1)}{\|\bar{P}(\lambda_1)\|}\\
				\vdots\\
				\frac{e^{iT\lambda_N}	\widehat{P}_0(\lambda_N)}{\|\bar{P}(\lambda_N)\|}
			\end{pmatrix}=\begin{pmatrix}
				\frac{\widehat{P}_N(\lambda_0)}{\|\bar{P}(\lambda_0)\|}\\
				\frac{\widehat{P}_N(\lambda_1)}{\|\bar{P}(\lambda_1)\|}\\
				\vdots\\
				\frac{\widehat{P}_n(\lambda_N)}{\|\bar{P}(\lambda_N)\|}
			\end{pmatrix}\qquad\text{ or }\qquad
			e^{i\varphi}\begin{pmatrix}
				e^{iT\lambda_0}	\widehat{P}_0(\lambda_0)\\
				e^{iT\lambda_1}	\widehat{P}_0(\lambda_1)\\
				\vdots\\
				e^{iT\lambda_N}	\widehat{P}_0(\lambda_N)
			\end{pmatrix}=\begin{pmatrix}
				\widehat{P}_N(\lambda_0)\\
				\widehat{P}_N(\lambda_1)\\
				\vdots\\
				\widehat{P}_N(\lambda_N)
			\end{pmatrix}.
		\end{equation}
		The latter equality reduces to
		\begin{equation}\label{interimPST1}
			\widehat{P}_N(\lambda_k)=e^{i\varphi}e^{iT\lambda_k}, \quad k=0,1,2,\dots,N
		\end{equation}
		since $\widehat{P}_0(x)=1$. Taking into account that $P_N$ is a polynomial with real coefficients and $\lambda_k$'s are real, we get that \eqref{interimPST1} can hold if and only if
		\begin{equation}\label{interimPST2}
			\widehat{P}_N(\lambda_k)=\pm 1
		\end{equation}
		and
		\begin{equation}\label{interimPST3}
			e^{i\varphi}e^{iT\lambda_k}=\pm 1.
		\end{equation}
		Next, it follows from  Proposition \ref{detFormula} that $\lambda_k$'s are the zeroes of $\widehat{P}_{N+1}$. Also, combining Theorem \ref{QPJfrac} and formula \eqref{JfracP} we conclude that the zeroes of $\widehat{P}_N$ and $\widehat{P}_{N+1}$ interlace, which implies that \eqref{interimPST2} is equivalent to \eqref{PST2}. Then, \eqref{interimPST3} becomes \eqref{PST1}.
	\end{proof}	
	
	At the first glance, the conditions \eqref{PST1} and \eqref{PST2} are not so easy to check. However, if one looks at \eqref{PST2} carefully, one can recognize one of the conditions that we have already seen before in Theorem \ref{theorem1}.
	Before we can proceed with that, let's introduce some concepts from the theory of symmetric matrices. Namely, we say that a Jacobi matrix $H_{[0,N]}$ is persymmetric or mirror symmetric if it satisfies the following relations
	\[
	a_k = a_{N-k},  \quad k= 0, 1, 2, \dots N, \quad 
	b_{n} = b_{N-1-n}, \quad n= 0, 1, 2, \dots N-1,
	\]
	which can also be expressed in the following way
	\[
	{H_{[0,N]}} ={ R}{H_{[0,N]}}{ R},  
	\]
	where the matrix ${R}$, the mirror reflection matrix, is
	\[
	{R}=\begin{pmatrix}
		0 & 0 & \dots & 0 & 1    \\
		0 & 0 & \dots  & 1 & 0  \\
		\vdots  & \vdots & \iddots & \vdots & \vdots      \\
		0 & 1&\dots&0&0\\
		1 & 0 &  \dots & 0 &0  \\
	\end{pmatrix}.
	\]
	
	\begin{corollary}[\cite{VZh12}] The condition \eqref{PST2} is equivalent to the mirror symmetry of the underlying Jacobi matrix $H_{[0,N]}$.
	\end{corollary}
	\begin{proof}
		The condition \eqref{PST2} is equivalent to the fact that $\widehat{P}_N^2-1$ is divisible by $\widehat{P}_{N+1}$. Due to the definition of $\widehat{P}_k$ in terms of $P_k$, we get that $\widehat{P}_N^2-1$ is divisible by $\widehat{P}_{N+1}$ if and only if ${P}_N^2-b
		_0^2\dots b_{N-1}^2$ is divisible by ${P}_{N+1}$. The latter holds true if and only if $P_N/P_{N+1}$ has a palindromic $J$-fraction representation due to formula \eqref{JfracP} and Theorem \ref{theorem1}. Clearly, a $J$-fraction is palindromic if and only if the corresponding Jacobi matrix is mirror symmetric.
	\end{proof}
	
	To demonstrate how the perfect transfer works let us consider the Jacobi matrix 
	\[
	H_{[0,2]}=\begin{pmatrix}
		0&\frac{1}{\sqrt{2}}&0\\
		\frac{1}{\sqrt{2}}&0&\frac{1}{\sqrt{2}}\\
		0&\frac{1}{\sqrt{2}}&0
	\end{pmatrix}
	\]
	that corresponds to \eqref{cheb} when $n=2$, that is, $H_{[0,2]}$ is associated to
	\[
	\frac{2x^2-1}{2x(x^2-1)}= \frac{1}{x-\displaystyle{\frac{\frac{1}{2}}
			{x-\displaystyle{\frac{\frac{1}{2}}{x}}}}},
	\]
	from which we see that the eigenvalues of $H_{[0,2]}$ are $-1$, $0$, and $1$. As a result, we get that $\varphi=T=\pi$. The following diagram represents the magnitudes of the corresponding components of the vector $e^{itH_{[0,2]}}\cdot e_0$ for different values of $t$. Note that it starts with $e_0$ when $t=0$ and for $t=T=\pi$ we get $e^{i\varphi} e_2=-e_2$.
	
	\noindent\includegraphics[width=0.2\textwidth]{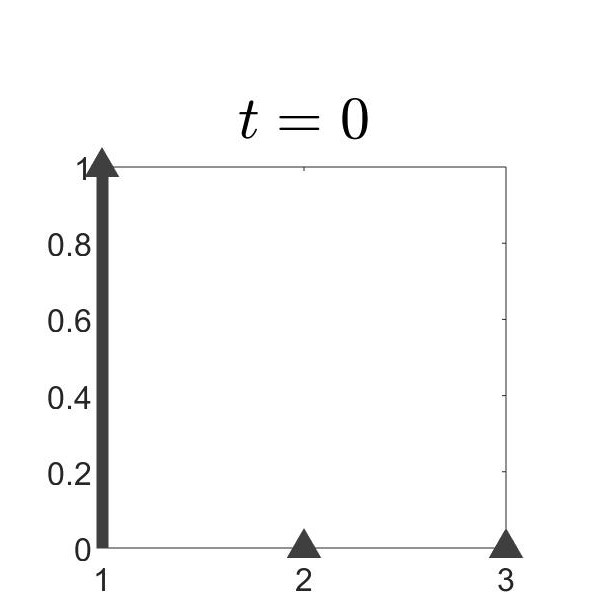}
	\includegraphics[width=0.2\textwidth]{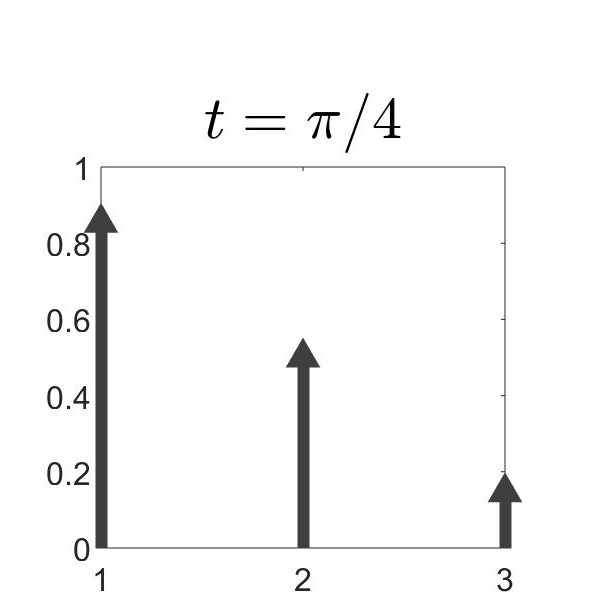}
	\includegraphics[width=0.2\textwidth]{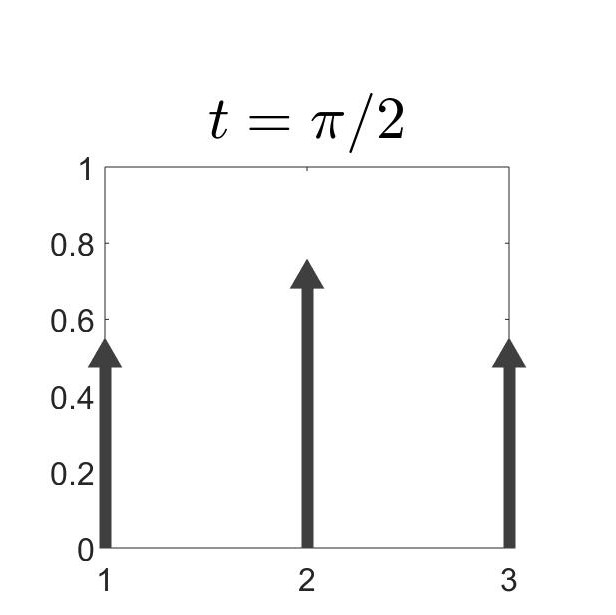}
	\includegraphics[width=0.2\textwidth]{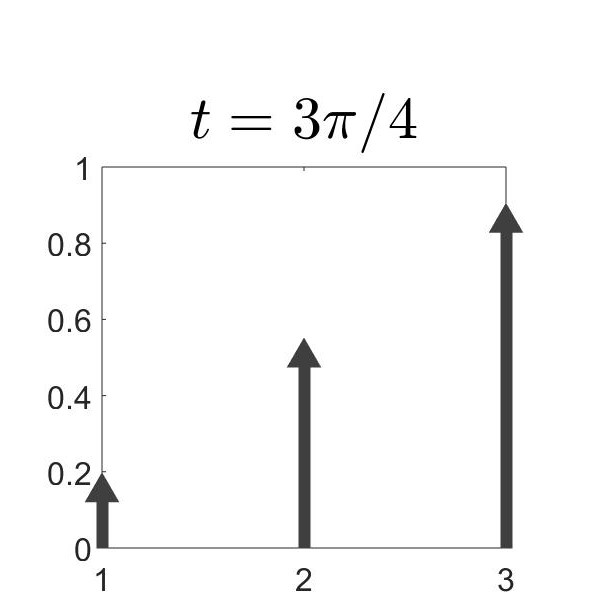}
	\includegraphics[width=0.2\textwidth]{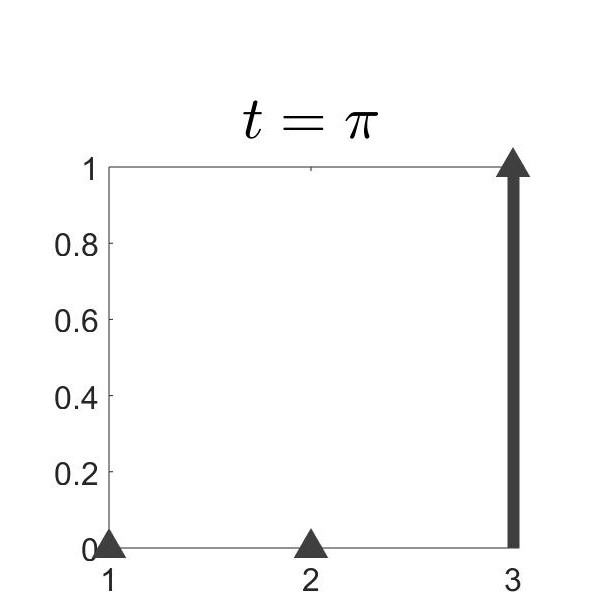}
	
	In principle, to design a Hamiltonian that realizes perfect state transfer one needs to come up with a Jacobi matrix that satisfies \eqref{PST1} and \eqref{PST2}. Accordingly, one should pick real numbers $\lambda_k$'s which satisfy \eqref{PST1}. Then, one needs to use the algorithm that reconstructs a persymmetric Jacobi matrix from its eigenvalues (for instance see \cite{G} or \cite{VZh12}), which completes the task. 
	
	\begin{remark} The fact that \eqref{PST2} is equivalent to the mirror symmetry of the underlying Jacobi matrix $H_{[0,N]}$ was pointed out by L. Vinet ane A. Zhedanov in \cite{VZh12} and they used some formulas from the theory of orthogonal polynomials that heavily relied on the positivity of the corresponding measure. Note that the positivity of the measure is equivalent to the interlacing of the zeros of $P_{N+1}$ and $Q_{N+1}$. However, one of the main ideas of this exposition is to show that the mirror symmetry of the corresponding Jacobi matrix is nothing else but the fact that the underlying $J$-fraction is palindromic, which is rather an algebraic property. To elaborate on this, in the next section we are going to consider the case of arbitrary rational function for which an analog of the Serret theorem still holds.
	\end{remark}
	
	\section{\textcolor{black}{The Serret theorem for the polynomial ring}}

	\textcolor{black}{In this section we show that the Serret theorem that was introduced at the beginning for positive integer numbers is still valid in the case of the ring of polynomials.}	
	
	To begin with, let us consider two polynomials $Q(x)=(x+1)(x+2)$ and $P(x)=(x-1)(x-2)(x-3)$, whose zeroes do not interlace. It is not so hard to see that the ratio $Q/P$ has the following continued fraction expansion{\color{black}
		\[
		\frac{Q(x)}{P(x)}=\frac{x^2+3x+2}{x^3-6x^2+11x-6}=\cfrac{1} {x-9+\cfrac{36}{x+\frac{8}{3}+\cfrac{\frac{10}{9}}{x+\frac{1}{3}}}}
		\]
		and that this continued fraction is not a $J$-fraction. Also, note that we can rewrite it as follows
		\[
		\frac{x^2+3x+2}{x^3-6x^2+11x-6}=\cfrac{1} {\mathfrak{p}_0(x)-\cfrac{1}{\mathfrak{p}_1(x)-\cfrac{1}{\mathfrak{p}_2(x)}}},
		\]
		where $\mathfrak{p}_0(x)=x-9$, $\mathfrak{p}_1(x)=-\frac{1}{36}(x+\frac{8}{3})$, and  $\mathfrak{p}_2(x)=\frac{162}{5}(x+\frac{1}{3})$. }
	Evidently, we can get the general situation by allowing the entries to be arbitrary polynomials. It is also clear that we can have some symmetry in such general continued fractions. For example, we can have the following
	\[\frac{Q(x)}{P(x)}=
	\frac{x^3+3x^2+x+2}{x^4+6x^3+10x^2+4x+3}=\cfrac{1} {x+3-\cfrac{1}{x^2+1-\cfrac{1}{x+3}}}.
	\]
	Note that $Q^2(x)=x^6+6x^5+11x^4+10x^3+13x^2+4x+4$. That is, $Q^2(x)-1=(x^2+1)P(x)$, which means that the Serret theorem is still valid.
	
	So, let's switch to the general case. To that end, consider two polynomials $P$ and $Q$ with complex coefficients. If $\deg Q<\deg P$ then the rational function $Q/P$ is a proper rational function. Next, applying the Euclidean algorithm to the pair  $(Q,P)$ leads to the continued fraction representation of the proper rational function $Q/P$ that can be written in the following way
	\begin{equation*}
		\frac{Q(x)}{P(x)}=\cfrac{1}{\mathfrak{p}_0(x)-\cfrac{1}{\mathfrak{p}_1(x)-\cfrac{1}{\ddots-\cfrac{1}{\mathfrak{p}_N(x)}}}},
	\end{equation*}
	where $\mathfrak{p}_0$, $\mathfrak{p}_1$, \dots, $\mathfrak{p}_N$ are some polynomials  with complex coefficients and none of which is identically zero. Such continued fractions are called {\it $P$-fractions} and one can associate some structured matrices to them \cite{DD04}. Next, we say that a finite $P$-fraction is palindromic if 
	\[
	\mathfrak{p}_{k}=\mathfrak{p}_{N-k}
	\]
	for all $k\in\{0, 1, 2, \dots, N\}$.
	Finally, as it was already noticed, the Serret theorem holds true for $P$-fractions.
	\begin{theorem}
		Let $P$ and $Q$ be monic polynomials with complex coefficients such that $\deg Q<\deg P$. The $P$-fraction representation of $Q/P$ is palindromic if and only if  $Q^2-1$ is divisible by $P$.
	\end{theorem}
	\begin{proof}
		The proof is basically a repetition of the algebraic part of Theorem \ref{theorem1}. Namely, in this case we need to omit the part where the equivalence of (i) and (ii) to the existence of $J$-fraction for $Q/P$ is shown. The only difference is that the recurrence relations to define $P_k$'s and $Q_k$'s for the $P$-fraction are the following:
		\begin{equation*}
			P_{k+1}(x) = P_{k}(x)\mathfrak{p}_k(x) -P_{k-1}(x), \quad k=0,1,2,\dots, N
		\end{equation*}
		and \begin{equation*}
			Q_{k+1} (x)= Q_{k}(x)\mathfrak{p}_k(x) - Q_{k-1}(x),\quad k=0,1,2,\dots, N,
		\end{equation*}
		where $P_{-1}(x) = 0$, $P_0 (x)= 1$, $Q_{-1} (x)= -1$, $Q_0(x) =0$ and $b_{-1}=1$. The rest easily follows by mimicking the corresponding argument of the proof of Theorem~\ref{theorem1}.
	\end{proof}

	\noindent {\bf Acknowledgments.} The authors are indebted to the anonymous referee for careful reading of the manuscript, pointing out numerous typos, and giving the suggestions that helped to improve the presentation of the material. They would also like to thank the UConn Young Scholars Senior Summit Program under which the study leading to this manuscript was initiated. M.D. was supported in part by the NSF DMS grant 2008844 and by the University of Connecticut Research Excellence Program.

\end{document}